\documentclass[12pt,oneside,english]{amsart}
\usepackage[T1]{fontenc}
\usepackage[latin9]{inputenc}
\usepackage{babel}
\usepackage{float}
\usepackage{amstext}
\usepackage{amsthm}
\usepackage{amssymb}
\usepackage{setspace}
\usepackage[unicode=true,pdfusetitle,
 bookmarks=true,bookmarksnumbered=false,bookmarksopen=false,
 breaklinks=false,pdfborder={0 0 1},backref=false,colorlinks=false]
 {hyperref}

\makeatletter

\newcommand{\mathcircumflex}[0]{\mbox{\^{}}}

\floatstyle{ruled}
\newfloat{algorithm}{tbp}{loa}
\providecommand{\algorithmname}{Algorithm}
\floatname{algorithm}{\protect\algorithmname}

\numberwithin{equation}{section}
\numberwithin{figure}{section}
\theoremstyle{plain}
\newtheorem{thm}{\protect\theoremname}
\theoremstyle{plain}
\newtheorem{lem}[thm]{\protect\lemmaname}
\theoremstyle{plain}
\newtheorem{cor}[thm]{\protect\corollaryname}
\theoremstyle{plain}
\newtheorem{prop}[thm]{\protect\propositionname}
\theoremstyle{definition}
\newtheorem*{problem*}{\protect\problemname}

\usepackage{fullpage}
\usepackage{amsfonts, amssymb}
\usepackage{pdfsync}
\usepackage{tikz, tikz-qtree}
\usepackage{mathtools}
\usepackage{MnSymbol}
\DeclareMathOperator{\gcol}{gcol}
\DeclareMathOperator{\col}{col}
\DeclareMathOperator{\ra}{ra}

\usepackage{float}

\usepackage{algorithm, algorithmic}

\makeatother

\providecommand{\corollaryname}{Corollary}
\providecommand{\lemmaname}{Lemma}
\providecommand{\problemname}{Problem}
\providecommand{\propositionname}{Proposition}
\providecommand{\theoremname}{Theorem}

\begin{document}
\title{Turn Skipping and the Game Coloring Number}
\author{Jason Guglielmo and Samuel Tinlin}
\date{\today}
\begin{abstract}
The game coloring number $\gcol(G)$ of a graph $G$ is a two player
competitive variant of the coloring number. We introduce the preordered
game coloring number to study the consequences of either player skipping
any number of turns. In particular, we show that neither player can
improve their performance by doing so. We use this result to show
that for any induced subgraph $H\subset G,$ $|G-H|=k$ implies the
tight bound $\gcol(H)\leq\gcol(G)-2k$ and if $d_{G}(x)\leq\gcol(G-x)$,
then $\gcol(G-x)\leq\gcol(G)-1$. 
\end{abstract}

\maketitle

\section{Introduction }

Let $G=\left(V,E\right)$ be a graph and let $\tau$ be a linear ordering
of $k$ vertices of $G$ for some $k\leq\left|G\right|$, that is,
$v_{1}\leq_{\tau}v_{2}\leq_{\tau}\cdots\leq_{\tau}v_{k}$ for $v_{1},\dots,v_{k}\in G$.
We identify such a linear ordering as the $k$-permutation $\tau=\left(v_{1},\dots,v_{k}\right)$.
If $k=0$, then we write the empty permutation $\tau=\left(\right)$,
and if $k=|G|$ then we call $\tau$ a \emph{complete permutation}
of $G.$ The \emph{range }of $\tau$ is defined as $\ra\left(\tau\right)=\left\{ v_{1},\dots,v_{k}\right\} $.
For some other $\ell$-permutation $\tau'=\left(w_{1},\dots,w_{\ell}\right)$
with $v_{i}\neq w_{j}$ for each $i,j$, we define the operation of
\emph{concatenation} as $\tau\mathcircumflex\tau'=\left(v_{1},\dots,v_{k},w_{1},\dots,w_{\ell}\right)$.
If $k<|G|$ then the set of all $k$-permutations of $G$ is denoted
by $\Pi_{k}\left(G\right)$; else if $k=|G|$ then it is denoted by
$\Pi\left(G\right)$. 

Now suppose $\tau$ is a complete permutation of $G=\left(V,E\right)$.
For any $v\in V$, we define the \emph{out-neighborhood} of $v$ in
$G$ with respect to $\tau$ as $N_{G}^{+}\left(\tau,v\right)=\left\{ u\in V\mid v>_{\tau}u\right\} $
and the \emph{closed out-neighborhood} as $N_{G}^{+}\left[\tau,v\right]=N_{G}^{+}\left(\tau,v\right)\cup\left\{ v\right\} $.
The \emph{coloring number of }$G$ \emph{with respect to }$\tau$,
denoted $\col\left(G,\tau\right)$, is defined as
\[
\col\left(G,\tau\right)=\max_{v\in V}\left|N_{G}^{+}\left[\tau,v\right]\right|.
\]
We then define the \emph{coloring number of $G$, }denoted $\col\left(G\right)$,
as
\[
\col\left(G\right)=\min_{\tau\in\Pi\left(G\right)}\left\{ \col\left(G,\tau\right)\right\} .
\]

Our focus in this paper is on a competitive variant of the coloring
number, called the game coloring number. The \emph{ordering game}
on $G$ is defined as follows: two players, Alice and Bob, take turns
choosing vertices from $G$ that have not yet been chosen to form
a complete permutation $\tau\in\Pi\left(G\right)$. Given some $k$-permutation
$\sigma$, we call $\sigma$ a \emph{preordering} of $G$. The \emph{$\sigma$-preordered
game }(or \emph{$\sigma$-game}) of $G$ is played in the same way
as the ordering game, except Alice and Bob take turns choosing vertices
from $V'=V-\ra\left(\sigma\right)$ and adjoining these vertices on
the preordering $\sigma$ to form some $\tau\in\Pi\left(G\right)$.
If no preordering is specified, we call it the \emph{preordered game.}

Suppose the $\sigma$-game has been played to form some $\tau\in\Pi\left(G\right)$.
The \emph{score} $s$ of the $\sigma$-game is defined as 
\[
s=\col\left(G,\tau\right).
\]

\noindent Alice is trying to force the smallest score possible, while
Bob is trying to force the largest score possible. We define the \emph{$\sigma$-game
coloring number}, denoted $\sigma\text{-}\gcol\left(G\right)$, to
be the least $s$ such that Alice has a strategy to obtain a score
of at most $s$ in the $\sigma$-game of $G$, regardless of how Bob
plays. 

Informally, a strategy for either player is a function that determines
how they should play at any given turn in the game. Let $\rho=\left(v_{1},\dots,v_{k}\right)$
be a $k$-permutation of some graph $G$. A \emph{strategy} $S$ is
a function from $\bigcup_{k=0}^{n}\Pi_{k}\left(G\right)$ to $V(G)$
defined by $S\left(\rho\right)=x$, where $x\in V\setminus\ra\left(\rho\right)$. 

If Alice begins the ordering game, then we call it the \emph{Alice
ordering} \emph{game }(or simply the ordering game)\emph{. }If Bob
begins the ordering game, then we call it the \emph{Bob ordering game}.\emph{
}Suppose $\sigma=(v_{1},...v_{m})$ is a preordering of $G$. Then
$\sigma$ gives the position of the game after $m$ vertices have
been ordered. If $m$ is even (odd), then Alice (Bob) begins the game.
If $m$ is odd (even) and Alice (Bob) begins the game instead, we
call it the \emph{Alice (Bob) $\sigma$-game}. The \emph{Alice (Bob)
$\sigma$}-\emph{game coloring number }$\sigma\text{-}\gcol_{A}\left(G\right)$
$\left(\sigma\text{-}\gcol_{B}\left(G\right)\right)$ is the $\sigma$-game
coloring number for the Alice (Bob) $\sigma$-game. Note that the
ordering game is just the case where $\sigma=\left(\right)$, which
recovers the \emph{game coloring number}, denoted by $\gcol\left(G\right)$.

Is it possible that skipping a turn may actually allow Alice to reduce
the score? The first goal of this paper is to determine how the game
coloring number is affected if either Alice or Bob are allowed to
skip one or more turns. We use the preordered game to accomplish this.

\section{Monotonicity and Turn Skipping}

A critical result of Wu and Zhu \cite{Monotonicity} is that the game
coloring number is a monotonic parameter. In this context, a monotonic
parameter on a graph $G$ is one whose value cannot increase on a
subgraph $H\subset G$. We need to extend this result to the preordered
game. For completeness, we have included our adaptation of Wu and
Zhu's argument as the proof of Lemma \ref{Monotonicity Lemma}.
\begin{lem}
Let $G=\left(V,E\right)$ be a graph, and let $\sigma=\left(v_{1},\cdots,v_{m}\right)$
be a preordering of $G$. If $H$ is a subgraph of $G$ with $\ra\left(\sigma\right)\subseteq V\left(H\right)$,
then $\sigma\text{-}\gcol\left(H\right)\leq\sigma\text{-}\gcol\left(G\right)$.\label{Monotonicity Lemma}
\end{lem}

\begin{proof}
We argue by induction on $\left|G\right|-\left|H\right|$. If $\left|G\right|-\left|H\right|=0$,
then $H:=G$, so we are done. Else, suppose $\left|G\right|-\left|H\right|>0$.
Then there exists $x\in V\setminus V(H)$. Let $H'=G-x$. By an obvious
induction argument, it suffices to show that $\sigma\text{-}\gcol\left(H'\right)\leq\sigma\text{-}\gcol\left(G\right)$.

For simplicity, rename $H'$ as $H$. Let $\sigma\text{-}\gcol\left(G\right)=s$.
By definition of the $\sigma$-game coloring number, Alice has a strategy
$S'_{A}$ for playing the $\sigma$-game on $G$ that results in a
score of at most $s$ regardless of how Bob plays. We will construct
a strategy $S_{A}$ for Alice to play on the $\sigma$-game on $H$
that results in a score of at most $s$. To do so, we consider two
$\sigma$-games: the real game on $H$ between Alice and Bob, and
an imaginary game on $G$ that Alice is playing against herself.

Before the first play of the game, let $\tau=\rho'=\sigma$. Else,
let $\tau$ represent the position before Alice's next move in the
real game and $\rho'$ represent the position after Alice's last move
in the imaginary game. Alice's strategy $S_{A}$ is determined in
Algorithm \ref{Monotonicity Algorithm}. In this algorithm, Alice's
strategy involves constructing $\rho$, where $\rho$ represents the
position in the imaginary game after Alice's interpretation of Bob's
last move $v$. If this is the first move of the game, we assume $v$
is undefined. We shall maintain that after Alice's construction of
$\rho$ (lines $5\text{-}11$) but before her redefinition of $\rho'$
(lines $12\text{-}26$), the following \emph{invariant} holds: 
\begin{equation}
\ra\left(\rho\right)-x=\begin{cases}
\ra\left(\tau\right), & \text{if }x\notin\ra\left(\rho\right)\\
\ra\left(\tau\right)+w\text{ for some }w\notin\ra\left(\tau\right), & \text{else.}
\end{cases}\label{eq: Invariant}
\end{equation}
Note that this invariant holds trivially after the first construction
of $\rho$ because $\rho=\sigma=\tau$ by Algorithm \ref{Monotonicity Algorithm}and
$x\notin\sigma$. There are three immediate consequences of this invariant:
\begin{align}
\ra\left(\tau\right) & \subseteq\ra\left(\rho\right)\label{eq: Containment}\\
\ra\left(\rho\right) & =\ra\left(\tau\right),\text{ if }x\notin\ra\left(\rho\right)\label{eq: Equality}\\
\left|\ra\left(\rho\right)-x\right|-\left|\ra\left(\tau\right)\right| & =1,\text{ else}.\label{eq: Size}
\end{align}

\begin{algorithm}
\begin{algorithmic}[1]

\REQUIRE $G=(V,E)$ and $H$ as described above, and  incomplete linear orderings $\tau$ and $\rho'$ of $H$ and $G$ respectively, with the invariants holding for both. If this is not the first play, then $v$ is the last move played by Bob.
\ENSURE Linear orderings $\tau'$ and $\rho$ of $H$ and $G$, respectively. Redefinition of the linear ordering $\rho'$.

\STATE {$\{$We check our termination conditions$\}$}
\IF {$\ra (\tau) = V(H)$}
\STATE end
\ENDIF

\STATE {$\{$We determine $\rho \}$}
\IF {$\tau = \sigma$}
\STATE $\rho = \sigma$
\ELSIF {$v \notin \ra (\rho')$}
\STATE $\rho := \rho' \string^ (v)$
\ELSE [$v$ is an illegal repeat]
\STATE $V' = V - \ra ( \rho' )$
\STATE $\rho := \rho' \string^ (y)$, where $y$ satisfies $\min_{y \in V'} d_{G} (y)$
\ENDIF

\STATE {$\{$We determine $\tau'$ and $\rho' \}$}
\IF {$V - {x} \subseteq \ra(\rho)$}
\STATE Alice chooses an unordered vertex $w$ and sets $\tau' := \tau \string^ (w)$
\ELSIF {$S_A ' (\rho) = a \neq x$}
\STATE $\tau' := \tau \string^ (a)$ \AND $\rho' := \rho \string^ (a)$
\ELSE [$x$ is an illegal option]
\STATE $V' = V - \ra (\rho) - x$ 
\STATE let $y$ satisfy $\min_{y \in V'} d_{G} (y)$
\IF {$V' - y = \emptyset$}
\STATE $\tau' := \tau \string^ (y)$ \AND $\rho' := \rho \string^ (x,y)$
\ELSE [there exists another unchosen vertex in $G$]
\STATE let $z = S_A' (\rho \string^ (x,y) )$
\STATE $\tau' := \tau \string^ (z)$ \AND $\rho' := \rho \string^ (x,y,z)$
\ENDIF
\ENDIF

\STATE {$\{$We again check our termination conditions$\}$}
\IF {$\ra (\tau ' ) = V(H)$}
\STATE end
\ENDIF
\end{algorithmic}

\caption{Alice's Monotonicity Algorithm\label{Monotonicity Algorithm}}
\end{algorithm}
We want Alice's strategy $S_{A}$ to interpret Bob's last move in
the real game $v$ (assuming this is not the first turn of the game)
as a move in the imaginary game so that $S'_{A}\left(\rho'\mathcircumflex\left(v\right)\right)$
will be her next move in both the imaginary and real game. Two problems
may arise. If $v$ has already been played in the imaginary game i.e.
$v\in\ra\left(\rho'\right)$, then Alice cannot interpret it as a
move in the imaginary game. We call this an \emph{illegal repeat.}
If $S'_{A}\left(\rho'\mathcircumflex\left(v\right)\right)=x$, Alice
cannot play $x$ in the real game because $x\notin V\left(H\right)$.
We call this an \emph{illegal option}. Note that from lines $19\text{-}28$
of the algorithm, an illegal option causes Alice to order $x$ in
the imaginary game. So (\ref{eq: Equality}) holds before an illegal
option, and (\ref{eq: Size}) holds after one. Also from lines $19\text{-}28$,
an illegal repeat can only occur after an illegal option has already
occurred.

Accounting for these illegal possibilities complicates Alice's strategy.
First, Alice constructs $\rho$ from $\rho'$. If this is the first
turn of the ordering game, then $\rho=\sigma$ according to lines
$6\text{-}7$ in Algorithm \ref{Monotonicity Algorithm}. Else if
$v$ is legal in the imaginary game, then Alice sets $\rho=\rho'\mathcircumflex\left(v\right)$
according to lines $8\text{-}9$. Else if $v$ is an illegal repeat,
then Alice sets $\rho=\rho'\mathcircumflex\left(y\right)$, where
$y\in V\left(H\right)-\ra\left(\rho'\right)$ is an unordered vertex
of smallest degree. One such $y$ exists because of lines $2\text{-}4$.
This process is seen in lines $10\text{-}12$.

Next, Alice redefines $\rho'$ from $\rho$ and constructs $\tau'$
from $\tau$, where $\tau'$ represents the position after Alice's
last play in the real game. Since Alice has constructed $\rho$ but
not yet redefined $\rho'$, we may invoke (\ref{eq: Invariant}).
If $V-x\subseteq\ra\left(\rho\right)$, then there is at most one
vertex in $V-x-\ra\left(\tau\right)$ by (\ref{eq: Size}). Since
the termination condition in lines $2\text{-}4$ was passed, there
is exactly one such vertex, call it $w$. Alice sets $\tau'=\tau\mathcircumflex\left(w\right)$
according to lines $15\text{-}16$. Else if $S'_{A}\left(\rho\right)=a\neq x$,
then Alice sets $\rho'=\rho\mathcircumflex\left(a\right)$ and $\tau'=\tau\mathcircumflex\left(a\right)$
according to lines $15\text{-}16$. We know such an $a$ can be played
in both the imaginary and real games by (\ref{eq: Size}). Else we
have an illegal option, so she will again determine the unchosen vertex
$y\in V\left(H\right)-\ra\left(\rho'\right)$ of smallest degree.
Recall, since the condition of lines $2\text{-}4$ was passed, there
is at least one unordered vertex in the real game. In particular,
$y$ is unordered in the real game because of (\ref{eq: Size}). So
if there are no vertices left to order in the imaginary game, then
Alice sets $\rho'=\rho\mathcircumflex\left(x,y\right)$ and $\tau'=\tau\mathcircumflex\left(y\right)$
according to lines $22\text{-}23$. Else letting $z=S'_{A}\left(\rho\mathcircumflex\left(x,y\right)\right)$,
she sets $\rho'=\rho\mathcircumflex\left(x,y,z\right)$ and $\tau'=\tau\mathcircumflex\left(z\right)$,
as seen in lines $24\text{-}26$. 

We must now check that our invariant holds after Bob's next play on
the real game, and Alice's next construction of $\rho$ in the imaginary
game. If Bob has no other moves to play, then we are done. Else, suppose
Bob sets $\tau=\tau'\mathcircumflex\left(b\right)$ for some $b\in V\left(H\right)-\ra\left(\tau'\right)$.
Let $\rho_{1}$ be the $\rho$ from before Bob's play of $b$, and
$\rho_{2}$ be the $\rho$ after Bob's play of $b$. Define $\tau_{1}$
and $\tau_{2}$ in the same way. Then we assume
\[
\ra\left(\rho_{1}\right)-x=\begin{cases}
\ra\left(\tau_{1}\right), & \text{if }x\notin\ra\left(\rho_{1}\right)\\
\ra\left(\tau_{1}\right)+w\text{ for some }w\notin\ra\left(\tau_{1}\right), & \text{else.}
\end{cases}
\]

First assume $x\notin\ra\left(\rho_{1}\right)$, so $\ra\left(\rho_{1}\right)-x=\ra\left(\tau_{1}\right)$.
If $V-x\subseteq\ra\left(\rho_{1}\right)$, then since $x\notin\ra\left(\rho_{1}\right)$,
$x$ is the only unordered vertex in the imaginary game, so we are
done. Else if $S'_{A}\left(\rho_{1}\right)=a\neq x$, we set $\rho'=\rho_{1}\mathcircumflex\left(a\right)$
and $\tau'=\tau_{1}\mathcircumflex\left(a\right)$. So $\tau_{2}=\tau_{1}\mathcircumflex\left(a,b\right)$.
Since $x\notin\ra\left(\rho'\right)$, $b$ cannot be an illegal repeat
so $\rho_{2}=\rho_{1}\mathcircumflex\left(a,b\right)$ and we get
$\ra\left(\rho_{2}\right)-x=\ra\left(\tau_{2}\right)$ with $x\notin\ra\left(\rho_{2}\right)$.
Else $S'_{A}\left(\rho_{1}\right)=x$. Letting $y$ and $V'$ be as
in lines $20\text{-}21$, if $V'-y=\emptyset$, then the game is finished
before the construction of $\rho_{2}$ and $\tau_{2}$. Else letting
$z$ be as in line $25$, we have $\rho'=\rho_{1}\mathcircumflex\left(x,y,z\right)$
and $\tau'=\tau_{1}\mathcircumflex\left(z\right)$. So $\tau_{2}=\tau_{1}\mathcircumflex\left(z,b\right)$.
If $b\neq y$, then $\rho_{2}=\rho_{1}\mathcircumflex\left(x,y,z,b\right)$
so $\ra\left(\rho_{2}\right)-x=\ra\left(\tau_{2}\right)+y$ with $x\in\ra\left(\rho_{2}\right)$.
Else $b=y\in\ra\left(\rho'\right)$ so we have an illegal repeat.
Letting $y'$ be as in line $12$, we have $\rho_{2}=\rho_{1}\mathcircumflex\left(x,y,z,y'\right)$.
So $\ra\left(\rho_{2}\right)-x=\ra\left(\tau_{2}\right)+y'$ with
$x\in\ra\left(\rho_{2}\right)$.

Finally assume $x\in\ra\left(\rho_{1}\right)$. If $V-x\subseteq\ra\left(\rho'\right)$,
then all vertices have been ordered in the imaginary game and we are
done. Else since $x$ has already been played, we must have $S'_{A}\left(\rho_{1}\right)=a\neq x$.
So we set $\rho'=\rho_{1}\mathcircumflex\left(a\right)$ and $\tau'=\tau_{1}\mathcircumflex\left(a\right)$,
and thus $\tau_{2}=\tau_{1}\mathcircumflex\left(a,b\right)$. If $b\notin\ra\left(\rho'\right)$,
then $\rho_{2}=\rho_{1}\mathcircumflex\left(a,b\right)$. However
since $x\in\ra\left(\rho_{1}\right)$, there exists $y\in\ra\left(\rho'\right)-\ra\left(\tau_{2}\right)$
by line $26$. Thus we have $\ra\left(\rho_{2}\right)-x=\ra\left(\tau_{2}\right)+y$
with $x\in\ra\left(\rho_{2}\right)$. Else $b\in\ra\left(\rho'\right)$,
so letting $y$ be as defined in line $12$, we have $\rho_{2}=\rho_{1}\mathcircumflex\left(a,y\right)$.
Therefore, $\ra\left(\rho_{2}\right)-x=\ra\left(\tau_{2}\right)+y$
with $x\in\ra\left(\rho_{2}\right)$.

We have shown that Algorithm \ref{Monotonicity Algorithm} maintains
(\ref{eq: Invariant}). Now we must show that if Alice plays according
to this algorithm, the final score in the real game will be at most
the final score in the imaginary game. Let $\tau$, $\tau'$, $\rho$,
and $\rho'$ be as described above. For any $u\in V\left(H\right)$,
let $s\left[u\right]$ be the score of $u$ in the $\sigma$-game
on $H$. Similarly, for any $w\in V\left(G\right)$, let $s'\left[w\right]$
be the score of $w$ in the $\sigma$-game on $G$. By assumption,
$s'\left[w\right]\leq s$ for all $w\in V\left(G\right)$. It suffices
to show that $s\left[u\right]\leq s$ for all $u\in V\left(H\right)$. 

Let $u\in V\left(H\right)$. First, suppose Alice just chose $u$
in $\tau'$. From lines $18,23,\text{ and }26$ of Algorithm \ref{Monotonicity Algorithm},
we know $u$ is the last vertex in $\rho'$. Also by (\ref{eq: Containment}),
all vertices ordered before $u$ in $\tau'$ have been ordered before
$u$ in $\rho'$. Thus, $s\left[u\right]\leq s'\left[u\right]\leq s$.
Next, suppose Bob just chose $u$ in $\tau$. If $u\notin\ra\left(\rho'\right)$,
then by line $9$ of the algorithm, $u$ is the last played vertex
in $\rho$. Then again by (\ref{eq: Containment}), we know $s\left[u\right]\leq s'\left[u\right]\leq s$. 

Else, $u$ was chosen in an earlier turn in $\rho$ than in $\tau$.
This can only happen if $u$ is a vertex of minimal degree amongst
unchosen vertices. Let $u^{*}$ denote the last vertex played in the
game. It suffices to show that the following chain of inequalities
hold:
\[
s\left[u\right]\leq d_{H}\left[u\right]\leq d_{G}\left[u\right]\leq d_{G}\left[u^{*}\right]=s'\left[u^{*}\right]\leq s.
\]

The first inequality holds because the backward neighbors of $u$
in $H$ with respect to $\tau$ contribute to the degree of $u$.
The second inequality holds because every neighbor of $u$ in $H$
is also a neighbor in $G$. The third inequality holds by minimal
choice of $u$ in line $20$. The equality holds because, since $u^{*}$
is the last vertex in either $\rho'$ or $\rho$, all of its neighbors
are backward neighbors. The last inequality holds by assumption.

Hence, in all cases, $s\left[u\right]\leq s$ for all $u\in V\left(H\right)$.
Therefore, by definition of the $\sigma$-game coloring number, $\sigma\text{-}\gcol\left(H\right)\leq s=\sigma\text{-}\gcol\left(G\right)$.
$\qedhere$
\end{proof}
Letting $\sigma$ be an empty preordering, we get the original monotonicity
result of Wu and Zhu. We will use the following Corollary.
\begin{cor}
Let $\sigma=\left(v_{1},\dots,v_{m}\right)$ be a preordering of $G=\left(V,E\right)$.
Let $H$ be a subgraph of $G$ with $\ra\left(\sigma\right)\subseteq V\left(H\right)$.
If $m$ is odd, then $\sigma\text{-}\gcol_{A}\left(H\right)\leq\sigma\text{-}\gcol_{A}\left(G\right)$.\label{Monotonicity Corollary}
\end{cor}

\begin{proof}
Let $G'=\left(V\cup\left\{ v\right\} ,E\right)$ and $H'=\left(V\left(H\right)\cup\left\{ v\right\} ,E\right)$,
where $v\notin V$. Also, let $\sigma'=\sigma\mathcircumflex\left(v\right)$.
Then, $\sigma'$ is a preordering of $G'$ with $\ra\left(\sigma'\right)\subseteq V\left(H'\right)$.
So, by Lemma \ref{Monotonicity Lemma}, 
\[
\sigma'\text{-}\gcol\left(H'\right)\leq\sigma'\text{-}\gcol\left(G'\right).
\]

Since $\left|\sigma'\right|$ is even by construction, Alice begins
the $\sigma'$-game on $H'$. However, because $v$ is an isolated
vertex, it has no effect on the $\sigma'$-game coloring number of
$H'$. That is, $\sigma'\text{-}\gcol\left(H'\right)=\sigma\text{-}\gcol_{A}\left(H\right)$.
Similarly, $\sigma'\text{-}\gcol\left(G'\right)=\sigma\text{-}\gcol_{A}\left(G\right)$.
Therefore, by substitution, we get
\[
\sigma\text{-}\gcol_{A}\left(H\right)\leq\sigma\text{-}\gcol_{A}\left(G\right).\qedhere
\]
\end{proof}
Now, we observe what happens if one of the players skips a single
turn in the ordering game.

\pagebreak{}
\begin{thm}
Let $G=\left(V,E\right)$ be a graph, and let $\sigma=\left(v_{1},\cdots,v_{m}\right)$
be a preordering of $G$.\label{Skipping any turn}
\end{thm}

\begin{enumerate}
\item If $m$ is even, then $\sigma\text{-}\gcol\left(G\right)\leq\sigma\text{-}\gcol_{B}\left(G\right)$.
\item If $m$ is odd, then $\sigma\text{-}\gcol_{A}\left(G\right)\leq\sigma\text{-}\gcol\left(G\right)$.
\end{enumerate}
\begin{proof}
$\left(1\right)$ Let $G'=\left(V',E\right)$, where $V'=V\cup\left\{ v\right\} $
for some isolated $v\notin V$. Clearly, $G\subset G'$ and $\ra\left(\sigma\right)\subset V'$.
So by Lemma \ref{Monotonicity Lemma}, $\sigma\text{-}\gcol\left(G\right)\leq\sigma\text{-}\gcol\left(G'\right)$.
It suffices to prove the following string of inequalities:
\[
\sigma\text{-}\gcol\left(G\right)\leq\sigma\text{-}\gcol\left(G'\right)\leq\left(\sigma\text{\textasciicircum}\left(v\right)\right)\text{-}\gcol\left(G'\right)=\sigma\text{-}\gcol_{B}\left(G\right).
\]

The first inequality holds by monotonicity. The second inequality
holds by the minimality of the $\sigma$-game coloring number. The
$\sigma$-game coloring number on $G'$ is witnessed by an optimal
strategy that Alice has to minimize the score, regardless of how Bob
plays. This strategy may not include choosing the vertex $v$ first,
so if she were to choose that vertex first, the score could increase.
Therefore, since the $\sigma\mathcircumflex\left(v\right)$-game can
be thought of as forcing Alice to choose $v$ first in the $\sigma$-game,
the score she can obtain in the $\sigma\mathcircumflex\left(v\right)$-game
cannot be smaller than the score she could normally obtain in the
ordering game.

As for the last equality, if Alice plays the isolated vertex first
on $G'$, she has not chosen any vertices in $G$. In other words,
the $\sigma\mathcircumflex\left(v\right)$-game on $G'$ is equivalent
to the Bob $\sigma$-game on $G$. Therefore, we are done.

$\left(2\right)$ Let $G'=\left(V',E\right)$, where $V'=V\cup\left\{ v\right\} $
for some isolated $v\notin V$. Clearly, $G\subset G'$ and $\ra\left(\sigma\right)\in V'$.
So by Corollary \ref{Monotonicity Corollary}, $\sigma\text{-}\gcol_{A}\left(G\right)\leq\sigma\text{-}\gcol_{A}\left(G'\right)$.
Again, following a similar argument as above, we have
\[
\sigma\text{-}\gcol_{A}\left(G\right)\leq\sigma\text{-}\gcol_{A}\left(G'\right)\leq\sigma\mathcircumflex\left(v\right)\text{-}\gcol_{B}\left(G'\right)=\sigma\text{-}\gcol\left(G\right).\qedhere
\]
\end{proof}
This theorem shows what happens when either Alice or Bob skip some
number of turns on the ordering game on $G=\left(V,E\right)$. Suppose
Alice is allowed to skip $k$ many turns in the ordering game. We
induct on $k$ to show that Alice cannot achieve a better score by
doing this. The case for $k=0$ is trivial, so suppose $k\geq1.$
Suppose Bob has an optimal strategy $S'_{B}$ that guarantees Alice
cannot reduce the score on the ordering game by skipping $\left(k-1\right)$
turns. Now let $\sigma=(v_{1},...v_{m})$ be the preordering which
gives the position of the game immediately before Alice skips her
$k$th turn. Then the ordering game resumes as the Bob $\sigma$-game,
with Bob using his optimal strategy $S''_{B}$ on this game. Bob's
overall strategy is given by
\[
S_{B}\left(v\right)=\begin{cases}
S'_{B}\left(v\right) & \text{if }v\in\ra\left(\sigma\right)\\
S''_{B}\left(v\right) & \text{else.}
\end{cases}
\]

By the inductive hypothesis, $\gcol\left(G\right)\leq\sigma\text{-}\gcol\left(G\right)$.
Since Bob is playing optimally on the Bob $\sigma$-game, we have
$\gcol\left(G\right)\leq\sigma\text{-}\gcol\left(G\right)\leq\sigma\text{-}\gcol_{B}\left(G\right)$
by Theorem \ref{Skipping any turn}. Therefore Alice cannot reduce
the score on the ordering game by skipping any number of turns. Following
a similar argument, if Bob is allowed to skip $\ell$ many turns,
$\sigma\text{-}\gcol_{A}\left(G\right)\leq\gcol\left(G\right)$ so
Bob cannot increase the score on the ordering game by skipping any
number of turns.

\section{Bounds on Induced Subgraphs}

In this section we show that for any graph $G$ and any $x\in V(G),$
$\gcol(G)\leq\gcol(G-x)+2$. We further present a construction for
which removing $k$ vertices from $G$ lowers the game coloring number
of that graph by $2k$, demonstrating that this bound is tight with
respect to induced subgraphs of any size. We begin with a lemma necessary
to the main result.
\begin{lem}
\label{Turn Skip Lemma}Let $G$ be a graph such that $\gcol_{A}(G)=s$.
Then $s\leq\gcol_{B}(G)\leq s+1.$
\end{lem}

\begin{proof}
By Theorem \ref{Skipping any turn}, we know $s\leq\gcol_{B}\left(G\right)$.
So, it suffices to show $\gcol_{B}\left(G\right)\leq s+1$. 

Let $x\in V(G)$ be the vertex Bob marks as the first move in the
Bob-ordering game on $G$ and let $G'=G-x$. By Lemma \ref{Monotonicity Lemma},
we know $\gcol_{A}\left(G'\right)\leq\gcol_{A}\left(G\right)=s$.
By definition, Alice has a strategy $S_{A}$ on the ordering game
on $G'$ that results in a score of at most $s$ regardless of how
Bob plays. 

Since Bob orders $x$ first, Alice can order all remaining vertices
according to $S_{A}$. Let $\tau=\left(x,v_{1},\dots,v_{n-1}\right)\in\Pi\left(G\right)$
be the permutation formed in accordance with $S_{A}$ and set $\tau'=\left(v_{1},\dots,v_{n-1}\right)$.
We know by definition of $S_{A}$ that for all $v\in V(G'),$
\[
\left|N_{G'}^{+}\left[\tau',v\right]\right|\leq s
\]
Furthermore, the addition of the single vertex $x$ at the beginning
of $\tau$ can only increase the number of outneighbors of any vertex
in $\tau'$ by at most one. Therefore for all $v\in V(G),$ 
\[
\left|N_{G}^{+}\left[\tau,v\right]\right|\leq\left|N_{G'}^{+}\left[\tau',v\right]\right|+1\leq s+1
\]
Therefore $\gcol_{B}\left(G\right)\leq s+1$. $\qedhere$
\end{proof}
\begin{lem}
\label{GCOL2 Graph} Let $G$ be a graph. Then $0\leq\gcol(G)-\gcol(G-x)\leq2$
for any $x\in V\left(G\right)$. 
\end{lem}

\begin{proof}
Let $x\in V\left(G\right)$ and suppose $\gcol(G-x)=s.$ It suffices
to show that Alice has a strategy on $G$ such that $\gcol(G)\leq s+2$.
Alice's strategy on $G$ should be exactly her strategy on $G-x$,
except she will immediately mark $x$. We now consider the Bob game
on $G-x$. By Lemma \ref{Turn Skip Lemma}, $s\leq\gcol_{B}(G-x)\leq s+1$,
so when the game on $G-x$ is finished every $v\in V(G-x)$ will have
at most $s$ backneighbors. Since Alice ordered $x$ first it has
no backneighbors. Therefore, joining $x$ to the beginning of the
ordering on $G-x$ yields a game coloring number of at most $s+2$,
if and only if $xv'\in E(G)$ for some $v'\in V(G-x)$ with $s$ backneighbors.
Thus $\gcol(G)\leq s+2$.
\end{proof}
The following theorem generalizes the result of Lemma \ref{GCOL2 Graph}
while providing a construction demonstrating that its bound is tight.
\begin{thm}
\label{Main Result} For every $n\geq3$ there exist graphs $G$,
$H\subset G$ with $|G|=n$ such that if $|G-H|=k$, then $\gcol(H)\leq\gcol(G)-2k.$
\end{thm}

\begin{proof}
We begin with the observation that $\gcol(K_{n}\lor\overline{K_{n-1}})=\gcol(K_{n-1}\lor\overline{K_{n-1}})+2$
for all $n\geq3.$ 

Fix $n\in\mathbb{N}$ and let $G=K_{n}\lor\overline{K_{n-1}}$. Then
$\gcol(G)\geq d_{G}(x)+1=2(n-1)+1=2n-1$ for any $x\in K_{n}$ since
Alice cannot order every vertex in $K_{n}$ before Bob orders every
vertex in $\overline{K_{n-1}}.$

Now fix $x\in K_{n}.$ Then $G-x=K_{n-1}\lor\overline{K_{n-1}}$ and
$\gcol(G-x)\leq d_{G}(x)=2n-3$ since Alice can order every vertex
in $K_{n}$ before Bob orders every vertex in $\overline{K_{n-1}}.$
By Lemma \ref{GCOL2 Graph}, we must have $\gcol(G)=\gcol(G-x)+2.$ 

So let $G=K_{n+3}\lor\overline{K_{n+2}}$. Then by the above,$\gcol(G)=2(n+2)+1=2n+5$.
Removing $n$ vertices from $K_{n+3}$ gives us $K_{3}\lor\overline{K_{n+2}}$
which has game coloring number 5. 
\end{proof}
\begin{prop}
\label{Remove Vertex}If $d_{G}(x)\leq\gcol(G-x)$, then $\gcol(G)\leq\gcol(G-x)+1$.
\end{prop}

\begin{proof}
Let $\gcol(G-x)=s$ and suppose $d_{G}(x)\leq s.$ Alice's strategy
will be to use the same strategy $S_{A}$ she used for $G-x$ on $G$,
never marking $x$ unless it is her last turn and she is forced to.
If Bob never orders $x,$then ordering it last yields $\gcol(G)\leq s+1$
since Alice's strategy on $G-x$ ensures that no $v\in V(G-x)$ has
more than $s-1$ backneighbors and $d_{G}(x)\leq s.$ Suppose Bob
marks $x$ at some point during the game. We can guarantee that Alice
can respond to Bob ordering $x$ with $S_{A}$ by treating it as a
skipped move, which by Theroem \ref{Skipping any turn} cannot itself
result in a higher game coloring number. Therefore it suffices to
check the consequences of actually ordering $x$. Let $\tau$ be the
permutation of $G$ created in accordance with $S_{A}$ and let $s(v)$
denote the number of backneighbors in $\tau$ for any $v$. Consider
any $v\in V(G-x)$ such that $s(v)=s-1.$ If $xv\notin E(G),$then
the ordering of $x$ does not matter. If $xv\in E(G)$, then whether
Bob marks $x$ before or after all $s-1$ backneighbors of $v$ have
been ordered, $s(v)=s.$ Thus, $\gcol(G)$ $\leq$ $\gcol(G-x)+1$. 
\end{proof}
This bound is also tight. Let $G=K_{3}\vee\overline{K_{2}}-e$ where
$e$ is some edge between a vertex in $K_{3}$ and a vertex in $\overline{K_{2}}.$
Then $\gcol(G)=4$ while $\gcol(K_{2}\vee\overline{K_{2}})=3.$ Unfortunately,
$d_{G}(x)\geq\gcol(G-x)$ does not imply that $\gcol(G)\geq\gcol(G-x)$.
An immediate counterexample is found by taking $C_{5}$ and adjoining
a vertex $x$ which is adjacent to any two adjacent vertices in $C_{5}.$

\pagebreak{}

\section{Further Considerations}

The concept of this paper came from an open question posed at the
end of \textbf{\cite{Asymmetric}}. This paper focused on the \emph{(a,b)-asymmetric
marking game}, which is a variant of the ordering game. In this game,
Alice and Bob still take turns putting vertices into a linear ordering,
but Alice orders $a$ vertices in a row before Bob orders $b$ vertices
in a row, for $a,b\geq1$. As this variant of the ordering game heavily
focuses on multiple turns being taken for each person, it is natural
to look into what happens when any of these turns are skipped. This
could be done by expanding the asymmetric marking game to a preordered
asymmetric marking game.

We now list some open problems. 
\begin{problem*}
For any graph $G$ or class of graphs $\mathcal{G},$ does there always
exist a turn that Alice can skip in the ordering game without increasing
the score?
\end{problem*}
\begin{problem*}
Does there exist a general graph construction as in Thereom \ref{Main Result}
such that $|G-H|=k$ implies $\gcol(H)\leq\gcol(G)-k$ for all such
$G$? 
\end{problem*}
\begin{problem*}
Does there exist a graph $G$ for which Alice increases $\gcol(G)$
every time she skips a turn?
\end{problem*}
\bibliographystyle{ieeetr}
\bibliography{Citations}

\end{document}